\documentclass[11pt,a4paper]{amsart}
\usepackage{amsthm,amsfonts,amsmath}
\usepackage[numbers,sort&compress]{natbib}
\usepackage[lmargin=40mm,rmargin=40mm,tmargin=40mm,bmargin=40mm,footskip=10mm]{geometry}
\usepackage[colorlinks=true,citecolor=black,linkcolor=black]{hyperref}
\newcommand{\doi}[1]{\href{http://dx.doi.org/#1}{doi:\texttt{#1}}}

\newcommand{\urlprefix}{}
\theoremstyle{plain}
\newtheorem{theorem}{Theorem}
\DeclareMathOperator{\MIS}{MIS}
\renewcommand{\baselinestretch}{1.2}
\begin{document}

\title[On the number of maximal independent sets in a graph]{On the
  number of maximal\\ independent sets in a graph}

\author{David R. Wood}


\thanks{Department of Mathematics and Statistics, The University of
  Melbourne, Melbourne, Australia
  (\texttt{woodd@unimelb.edu.au}). Supported by a QEII Fellowship from
  the Australian Research Council.}

\date{\today}

\begin{abstract} 
  Miller and Muller (1960) and independently Moon and Moser (1965)
  determined the maximum number of maximal independent sets in an
  $n$-vertex graph.  We give a new and simple proof of this result.
\end{abstract}

\maketitle

Let $G$ be a (simple, undirected, finite) graph. A set $S\subseteq
V(G)$ is \emph{independent} if no edge of $G$ has both its endpoints
in $S$. An independent set $S$ is \emph{maximal} if no independent set
of $G$ properly contains $S$. Let $\MIS(G)$ be the set of all maximal
independent sets in $G$. \citet{MM60} and \citet{MoonMoser}
independently proved that the maximum, taken over all $n$-vertex
graphs $G$, of $|\MIS(G)|$ equals
$$g(n):=\begin{cases}
  3^{n/3}& \text{ if }n\equiv0\pmod{3}\\
  4\cdot3^{(n-4)/3}& \text{ if }n\equiv1\pmod{3}\\
  2\cdot3^{(n-2)/3}& \text{ if }n\equiv2\pmod{3}\enspace.
\end{cases}$$ This result is important because $g(n)$ bounds the time
complexity of various algorithms that output all maximal independent
sets
\citep{BronKerbosch,LLR-SJC80,TIAS-SJC77,TTT-TCS06,JYP-IPL88,Eppstein-JGAA03,ELS}.
Here we give a new and simple proof of this upper bound on
$|\MIS(G)|$.

\begin{theorem}[\citep{MM60,MoonMoser}] 
  For every $n$-vertex graph $G$, $$|\MIS(G)|\leq g(n)\enspace.$$
\end{theorem}

\begin{proof}
  We proceed by induction on $n$.  The base case with $n\leq 2$ is
  easily verified.  Now assume that $n\geq3$.  Let $G$ be a graph with
  $n$ vertices.  Let $d$ be the minimum degree of $G$.  Let $v$ be a
  vertex of degree $d$ in $G$.  Let $N[v]$ be the closed neighbourhood
  of $v$.  If $I\in\MIS(G)$ then $I\cap N[v]\neq\emptyset$, otherwise
  $I\cup\{v\}$ would be an independent set. Moreover, if $w\in I\cap
  N[v]$ then $I\setminus\{w\}\in\MIS(G-N[w])$. Thus
$$|\MIS(G)|\leq\sum_{w\in N_G[v]} |\MIS(G-N_G[w])|\enspace,$$
Since $\deg(w)\geq d$ and $g$ is non-decreasing, by induction,
\begin{equation*}
  \label{Ind}
  |\MIS(G)|\leq(d+1)\cdot g(n-d-1)\enspace.
\end{equation*}
Note that
$$  4\cdot3^{(n-4)/3} \leq g(n) \leq  3^{n/3}\enspace.$$
If $d\geq 3$ then
$$|\MIS(G)|\leq(d+1)\cdot 3^{(n-d-1)/3}\leq 4\cdot 3^{(n-4)/3}\leq
g(n)\enspace.$$ If $d=2$ then
$$|\MIS(G)|\leq 3\cdot g(n-3)=g(n)\enspace.$$
If $d=1$ and $n\equiv1\pmod{3}$ then since $n-2\equiv2\pmod{3}$,
$$\MIS(G)\leq 2\cdot g(n-2)\leq 2\cdot2\cdot 3^{(n-2-2)/3}=
4\cdot3^{(n-4)/3}=g(n)\enspace.$$
If $d=1$ and $n\equiv0\pmod{3}$ then since $n-2\equiv1\pmod{3}$,
$$\MIS(G)\leq 2\cdot g(n-2)\leq 2\cdot4\cdot 3^{(n-2-4)/3}
<3^{n/3}=g(n)\enspace.$$ If $d=1$ and $n\equiv2\pmod{3}$ then since
$n-2\equiv0\pmod{3}$,
$$\MIS(G)\leq 2\cdot g(n-2)\leq 2\cdot3^{(n-2)/3}=g(n)\enspace.$$
This proves that $|\MIS(G)|\leq g(n)$, as desired.
\end{proof}



For completeness we describe the example by \citet{MM60} and
\citet{MoonMoser} that proves that Theorem~1 is best possible.  If
$n\equiv0\pmod{3}$ then let $M_n$ be the disjoint union of
$\frac{n}{3}$ copies of $K_3$.  If $n\equiv1\pmod{3}$ then let $M_n$
be the disjoint union of $K_4$ and $\frac{n-4}{3}$ copies of $K_3$.
If $n\equiv2\pmod{3}$ then let $M_n$ be the disjoint union of $K_2$
and $\frac{n-2}{3}$ copies of $K_3$.  Observe that $|\MIS(M_n)|=g(n)$.

Note that \citet{Vatter} independently proved Theorem 1, and gave a
connection between this result and the question, ``What is the largest
integer that is the product of positive integers with sum $n$?''
Also note that Dieter Kratsch proved that $|\MIS(G)|\leq 3^{n/3}$
using a similar proof to that presented here; see \citet[page
177]{Gaspers}. Thanks to the authors of \citep{Vatter,Gaspers} for
pointing out these references.


\def\cprime{$'$} \def\soft#1{\leavevmode\setbox0=\hbox{h}\dimen7=\ht0\advance
  \dimen7 by-1ex\relax\if t#1\relax\rlap{\raise.6\dimen7
  \hbox{\kern.3ex\char'47}}#1\relax\else\if T#1\relax
  \rlap{\raise.5\dimen7\hbox{\kern1.3ex\char'47}}#1\relax \else\if
  d#1\relax\rlap{\raise.5\dimen7\hbox{\kern.9ex \char'47}}#1\relax\else\if
  D#1\relax\rlap{\raise.5\dimen7 \hbox{\kern1.4ex\char'47}}#1\relax\else\if
  l#1\relax \rlap{\raise.5\dimen7\hbox{\kern.4ex\char'47}}#1\relax \else\if
  L#1\relax\rlap{\raise.5\dimen7\hbox{\kern.7ex
  \char'47}}#1\relax\else\message{accent \string\soft \space #1 not
  defined!}#1\relax\fi\fi\fi\fi\fi\fi}

\end{document}